\documentclass{amsart}

\usepackage[english]{babel}
\usepackage{csquotes}
\usepackage{textcomp}
\usepackage{varioref}
\usepackage[pdfusetitle]{hyperref}
\usepackage[nobysame,alphabetic,initials,msc-links]{amsrefs}
\DefineSimpleKey{bib}{how}

\renewcommand{\eprint}[1]{#1}
\BibSpec{misc}{%
  +{}{\PrintAuthors}  {author}
  +{,}{ \textit}      {title}
  +{,}{ }             {how}
  +{}{ \parenthesize} {date}
  +{,} { available at \eprint}        {eprint}
  +{,}{ available at \url}{url}
  +{,}{ }             {note}
  +{.}{}              {transition}
}

\usepackage[nameinlink, noabbrev, capitalize]{cleveref}
\usepackage{microtype}

\usepackage{amsfonts}
\usepackage{amssymb}
\usepackage{MnSymbol}
\usepackage{mathtools}
\usepackage{bbm}

\usepackage{graphicx}
\usepackage{multirow}
\usepackage{enumitem}
\usepackage{multicol}
\usepackage{float}
\usepackage{caption}
\usepackage{subcaption}
\usepackage[all]{xy}

\usepackage{amsthm}
\usepackage{thmtools}

\declaretheorem[style=definition,qed=$\heartsuit$,numberwithin=section]{definition}

\declaretheorem[style=definition,qed=$\diamondsuit$,sibling=definition]{example}

\declaretheorem[style=definition,qed=$\diamondsuit$,sibling=definition]{remark}

\theoremstyle{plain}
\newtheorem{theorem}[definition]{Theorem}
\newtheorem{proposition}[definition]{Proposition}
\newtheorem{lemma}[definition]{Lemma}
\newtheorem{corollary}[definition]{Corollary}

\DeclareMathOperator{\ima}{Im}

\DeclareMathOperator{\spn}{span}
\DeclareMathOperator{\Aut}{Aut}

\title{Sheaves of measures and KMS-weights on topological graph algebras}
\author{Jonas Eidesen}
\address{Department of Mathematics, University of Oslo, Norway}
\email{jonaeid@math.uio.no}
\date{October 5th 2023}

\numberwithin{equation}{section}

\pdfoutput = 1

\begin{document}

\begin{abstract}
    We show that the collection of regular Borel measures on a second-countable locally compact Hausdorff space has the structure of a sheaf. With this we give an alternate description of the pullback of a regular Borel measure along a local homeomorphism. We are able to use these tools to give a description of the KMS$_\beta$-weights for the gauge-action on the graph $C^*$-algebra of a second-countable topological graph in terms of sub-invariant measures on the vertex space of said topological graph.
\end{abstract}

\maketitle

\section{Introduction}

Topological graphs give a setting for studying a large family of $C^*$-algebras, among others it generalizes the existing theory of graph algebras and gives a new setting to study homeomorphism $C^*$-algebras. A lot of the theory was first developed by Katsura in \cites{Katsura2003}. Katsura associated a $C^*$-algebra to the quadruple $E=(E^0,E^1,r,s)$ where $E^0$ and $E^1$ are locally compact Hausdorff spaces, $s:E^1\rightarrow E^0$ is a local homeomorphism and $r:E^1\rightarrow E^0$ is continuous. It was shown by Yeend in \cites{Yeend2006, Yeend2007} that the $C^*$-algebras constructed from topological graphs admit a groupoid model. This is done by constructing the boundary path space $\partial E$ and constructing a monoidal action on this space from the natural numbers $\mathbb{N}$ via the \emph{backwards shift map} $\sigma:\partial E \setminus E^0 \rightarrow \partial E$. With this one can consider the Deaconu-Renault groupoid of this action.

Groupoid $C^*$-algebras have been a fruitful class of $C^*$-algebras when studying dynamical systems and KMS$_\beta$-states. Given a groupoid $\mathcal{G}$ and a continuous groupoid homomorphism $c:\mathcal{G}\rightarrow\mathbb{R}$, one is able to construct a $C^*$-dynamical system $(C^*(\mathcal{G}),\mathbb{R},\alpha^c)$ using Pontryagin duality. A lot of work was initiated by Renault in \cites{Renault1980} where he is able to give a description of the KMS$_\beta$-states for $\alpha^c$ in the case where $\mathcal{G}$ is a locally compact Hausdorff étale groupoid. He shows that each KMS$_\beta$-state for $\alpha^c$ restricts to a quasi-invariant probability measure on the unit space $\mathcal{G}^{(0)}$ with Radon-Nikodym derivative $e^{-\beta c}$. When the groupoid is principal he is also able to show that these KMS$_\beta$-states are in fact all the KMS$_\beta$-states for $\alpha^c$.

In \cites{Neshveyev2013} Neshveyev is able to generalize Renault's description of KMS$_\beta$-states for $\alpha^c$ on principal étale groupoids to the non-principal case. Christensen is able to further generalize Neshveyev's results to give a description of KMS$_\beta$-\emph{weights} for $\alpha^c$ in \cites{Christensen2023}. It is worth mentioning that there exist $C^*$-dynamical systems for which there does not exist any KMS$_\beta$-states, but there does exist KMS$_\beta$-weights. Hence, it is worth it to strive for describing not only the KMS$_\beta$-states of a $C^*$-dynamical system, but also the KMS$_\beta$-weights. We give an example of such a situation at the end of this paper, c.f.~\cref{example:no KMS states but KMS weights}.

Schafhauser initiated the study of tracial states on $C^*$-algebras associated to topological graphs in \cites{Schafhauser2018}. The $C^*$-algebra of a topological graph comes equipped with a natural action from the circle group, generally called the \emph{gauge-action}. Schafhauser was in particular interested in the tracial states which are gauge-invariant, and he was able to give a description of these tracial states both in terms of invariant probability measures on the boundary path space $\partial E$ and vertex-invariant probability measures on the vertex space $E^0$. Schafhauser was unable to completely characterize the gauge-invariant tracial states but conjectured that for free topological graphs, every tracial state should be gauge-invariant.

In \cites{Christensen2022} Christensen uses the program developed by Schafhauser along with his results from \cites{Christensen2023} to study tracial \emph{weights} on $C^*$-algebras associated to \emph{second-countable} topological graphs. He is in particular able to give a description of the gauge-invariant tracial weights, and proves Schafhauser's conjecture to be true.

We realized that the techniques developed by Schafhauser in \cites{Schafhauser2018} and Christensen in \cites{Christensen2022} can be used to study KMS$_\beta$-weights for the gauge-action on $C^*$-algebras associated to second-countable topological graphs.

In addition to studying the structure of KMS$_\beta$-weights for the gauge-action on the $C^*$-algebra of a second-countable topological graph, we became interested in the purely measure-theoretic questions that arose in our studies. In particular, there are a lot of results in functional analysis that are of the form "this set of functionals is in bijection with this set of measures". Riesz representation Theorem is the archetypical example of this type of result. In reality, we almost never use these measures explicitly, more often then not we only concern ourselves with the integral over these measures and work with these integrals. There should however be a way to work with these measures directly instead of constantly appealing to the duality between functionals and measures.

The motivating example that led us towards studying this deeper was the idea of pulling back measures. The existing idea has the following setup: let $\varphi:X\rightarrow Y$ be a local homeomorphism between locally compact Hausdorff spaces and $\mu$ be a regular Borel measure on $Y$. Then there exist a unique regular Borel measure $\varphi^*\mu$ on $X$ such that
\begin{equation}\label{eqn:old pullback of measure}
    \int_X f d\varphi^*\mu = \int_Y \sum_{x\in\varphi^{-1}(y)} f(x) d\mu(y)
\end{equation}
for all $f\in C_c(X)$.

It isn't immediately clear from this formula what the measure $\varphi^*\mu$ actually \emph{measures}. Our goal became to obtain a definition of the pullback which preserves the geometric intuition one would hope to have. In particular, since $\varphi$ is a \emph{local} homeomorphism one would expect that $\varphi^*\mu$ would \emph{locally} have the same behavior as $\mu$. This is made precise in \cref{prop:pullback of measure}.

We use this new description of the pullback (of a regular Borel measure) to give a complete description of the KMS$_\beta$-weights for the gauge action of the graph $C^*$-algebra of a second-countable topological graph.

\subsection{Outline}

In \cref{sec:prerequisites} we go through all the necessary prerequisites about sheaf-theory, groupoids and groupoid $C^*$-algebras as well as topological graphs. We particularly sketch the construction of the graph groupoid and the gauge-action on the graph $C^*$-algebra. The fact that the $C^*$-algebra of a topological graph admits a groupoid model (due to Yeend, c.f.~\cites{Yeend2006, Yeend2007}) allows us to use Renault's, Neshveyev's and Christensen's classification of KMS$_\beta$-states and -weights on groupoids, (c.f.~\cites{Renault1980, Neshveyev2013, Christensen2023}) to get a similar classification of KMS$_\beta$-states and -weights for the gauge-action on the $C^*$-algebra of a second-countable topological graph.

In \cref{sec:measures} we introduce the functors $\mathcal{M}$ and $\mathcal{M}_{\text{reg}}$ from the topology on a space $X$ (viewed as a poset-category) to the category of sets, sending an open subset to the collection of all Borel measures and regular Borel measures on that open subset respectively. We show that for arbitrary topological spaces, these functors have the structure of presheaves ($\mathcal{M}_{\text{reg}}$ has the structure of a sub-presheaf of $\mathcal{M}$). For second-countable spaces we show that $\mathcal{M}$ has the structure of a sheaf, and for second-countable locally compact Hausdorff spaces, $\mathcal{M}_{\text{reg}}$ has the structure of a sheaf. We use this to give a definition of the pullback of a regular Borel measure along a local homeomorphism, which coincides with the pullback as defined through \cref{eqn:old pullback of measure}, c.f.~\cref{prop:pullback of measure} and \cref{prop:uniqueness of pullback}.

In \cref{sec:quasi-invariant measures} we translate the results of Neshveyev and Christensen, (c.f.~\cites{Neshveyev2013, Christensen2023}) from the language of groupoid $C^*$-algebras to the language of topological graph $C^*$-algebras. In particular, we introduce \emph{$\beta$-quasi-invariant measures}, which are measures on the boundary path space of a second-countable topological graph that turn out to be the measures that correspond to the KMS$_\beta$-weights for the gauge action on the $C^*$-algebra of a second-countable topological graph, c.f.~\cref{prop:KMS weights graph}.

In \cref{sec:sub-inv measures} we introduce \emph{$\beta$-sub-invariant measures}, which are measures on the vertex space of a second-countable topological graph. We prove that there is an affine bijection between the $\beta$-quasi-invariant measures of a second-countable topological graph and the $\beta$-sub-invariant measures. Since the vertex space of a second-countable topological graph will in general be easier to get a grasp of then the boundary path space of a second-countable topological graph, this bijection allows one to, for example, compute the KMS-spectrum of the graph $C^*$-algebra (in regard to the gauge-action).

\medskip

To keep things more organized in this paper we mark the end of definitions with the symbol $\heartsuit$, and we end remarks and examples with the symbol $\diamondsuit$.

\subsection{Acknowledgments}

The contents of this paper is based on the authors master's thesis, written at the University of Oslo, supervised by Nadia Larsen.

\section{Prerequisites}\label{sec:prerequisites}

\subsection{Sheaves}

Since sheaf theory is not the most standard tool in functional analysis it will be useful to recall the definitions here. See \cites{Hartshorne1977} for a thorough introduction in the setting of algebraic geometry.

Let $X$ be a topological space and denote its topology by $\mathcal{T}$. A \emph{presheaf of sets on $X$} is a contravariant functor $\mathcal{F}:\mathcal{T}\rightarrow\mathbf{Sets}$ where $\mathcal{T}$ is to be understood as a poset-category ordered by inclusion. For $U,V\in\mathcal{T}$ the elements $s\in\mathcal{F}(V)$ are called \emph{sections} and if $U\subset V$ we call the function $\rho^V_U := \mathcal{F}(U\subset V)$ for a \emph{restriction}, and we use the notation $\rho^V_U(s) = s|_U$.

A \emph{sheaf of sets on $X$} is a presheaf of sets on $X$, $\mathcal{F}$, satisfying the two following axioms for any open subset $W\subset X$ and any open cover $\mathcal{U}$ of $W$:
\begin{itemize}
    \item[(\emph{Locality})] If $s,t\in\mathcal{F}(W)$ are such that $s|_U = t|_U$ for all $U\in\mathcal{U}$ we have that $s = t$.
    \item[(\emph{Gluing})] If we have a family of sections $\{s_U\}_{U\in\mathcal{U}}$ such that $s_U \in \mathcal{F}(U)$ and $s_U|_{U\cap V} = s_V|_{U\cap V}$ for every pair $U,V\in\mathcal{U}$, then there exists a section $s\in\mathcal{F}(W)$ such that $s|_U = s_U$ for all $U\in\mathcal{U}$.
\end{itemize}
It is clear that the gluing axiom gives the existence of a global section whenever there are local sections that agree on intersections, and the locality axiom gives uniqueness of this global section. A priori one might also believe that this global section is dependent on the open cover $\mathcal{U}$, it is however easy to show that this is not the case.

If $\mathcal{H}$ is also a (pre)sheaf of sets on $X$ such that $\mathcal{H}(U) \subset \mathcal{F}(U)$ and the diagram
\begin{equation*}
    \xymatrix{
        \mathcal{F}(V) \ar[r] & \mathcal{F}(U) \\
        \mathcal{H}(V) \ar[r] \ar@{^{(}->}[u] & \mathcal{H}(U) \ar@{^{(}->}[u]
    }
\end{equation*}
commutes for each pair of open subsets $U,V\subset X$ such that $U\subset V$, we say that $\mathcal{H}$ is a \emph{sub-(pre)sheaf} of $\mathcal{F}$. Note that the vertical arrows in this diagram are inclusions and the horizontal arrows are the restrictions in $\mathcal{F}$ and $\mathcal{H}$ respectively.

It is worth recalling that a (pre)sheaf can take values in any category, not only $\mathbf{Sets}$. A particularly useful example is when a (pre)sheaf takes values in $\mathbf{Ab}$, the category of abelian groups. In this case, one has access to the tools of homological algebra, thus it is possible to define sheaf-cohomology, which is a powerful tool in many settings.

\subsection{Étale groupoids and their \texorpdfstring{$C^*$} {C*}-algebras}

For a thorough introduction to this topic see the work done by Renault in \cites{Renault1980}. We define a groupoid $\mathcal{G}$ to be a small category where each morphism is invertible. We denote the range map (codomain map) by $r:\mathcal{G}\rightarrow\mathcal{G}$ and the source map (domain map) by $s:\mathcal{G}\rightarrow\mathcal{G}$. The set of objects of $\mathcal{G}$ is denoted by $\mathcal{G}^{(0)}$ and will be referred to as the \emph{unit space}. We will generally use the letters $x,y,z\in\mathcal{G}$ to denote general elements in the groupoid and the letters $u,v,w\in\mathcal{G}^{(0)}$ to denote units. For a unit $u\in\mathcal{G}^{(0)}$ we define $\mathcal{G}_u = s^{-1}(u)$ and $\mathcal{G}^u = r^{-1}(u)$, and denote by $\mathcal{G}_u^u = \mathcal{G}_u \cap \mathcal{G}^u$ the \emph{isotropy subgroup at $u$}. We will only consider groupoids which are second-countable locally compact Hausdorff and étale, where étale means that both the range map and source map are local homeomorphisms. To avoid cluttered terminology these groupoids will simply be referred to as \emph{étale groupoids}. An open subset $W$ of an étale groupoid $\mathcal{G}$ is called a \emph{bisection} if both $r|_W$ and $s|_W$ are injective. An important fact is that every étale groupoid has a countable base of bisections.

We construct a $*$-algebra from an étale groupoid $\mathcal{G}$ by looking at the complex vector space $C_c(\mathcal{G})$ with the convolution product $f * g$ of two functions $f,g\in C_c(\mathcal{G})$ defined by the equation
\begin{equation*}
    (f * g)(x) = \sum_{y\in\mathcal{G}^{r(x)}} f(y)\,g(y^{-1}x)
\end{equation*}
for any $x\in\mathcal{G}$, and involution defined by $f^*(x) = \overline{f(x^{-1})}$. The full $C^*$-norm on $C_c(\mathcal{G})$ is defined by the equation
\begin{equation*}
    \| f \| = \sup\{ \| \pi(f) \| \ | \ \pi \text{ is a } * \text{-representation of } C_c(\mathcal{G}) \}
\end{equation*}
for $f\in C_c(\mathcal{G})$, and the full groupoid $C^*$-algebra is defined as the completion of $C_c(\mathcal{G})$ in this norm, namely
\begin{equation*}
    C^*(\mathcal{G}) = \overline{C_c(\mathcal{G})}^{\|\,\cdot\,\|}.
\end{equation*}

\subsection{Dynamical systems}

For an introduction in $C^*$-dynamical systems see for example \cites{Bratteli1987}. A \emph{$C^*$-dynamical system} is a triple $(\mathcal{A},G,\alpha)$ with $\mathcal{A}$ being a $C^*$-algebra, $G$ a locally compact group and $\alpha$ a strongly continuous map $\alpha:G\rightarrow \Aut(\mathcal{A})$ satisfying $\alpha_e = id_\mathcal{A}$ and $\alpha_g \alpha_h  = \alpha_{gh}$, where $e\in G$ is the identity element and $g,h\in G$. By strongly continuous we mean that the map $g\mapsto \alpha_g(a)$ is norm-continuous for all $a\in\mathcal{A}$.

We will in particular be interested in the case where $G=\mathbb{R}$. In this case it will be possible to analytically continue the dynamics of $\alpha$ from $\mathbb{R}$ to $\mathbb{C}$. Hence, let $(\mathcal{A},\mathbb{R},\alpha)$ be a $C^*$-dynamical system. For $z\in\mathbb{C}$ with $\ima(z)\geq0$ we define the set $S(z)$ to be the horizontal strip $S(z) = \{w\in\mathbb{C} \ | \ \ima(w)\in[0,\ima(z)] \}$. When $\ima(z)\leq0$ we define $S(z)$ similarly: $S(z) = \{w\in\mathbb{C} \ | \ \ima(w)\in[\ima(z),0] \}$.

We define the set $D(\alpha_z)\subset\mathcal{A}$ to be the elements $a\in\mathcal{A}$ such that there exists a continuous function $f:S(z)\rightarrow\mathcal{A}$ which is \emph{analytic} on the interior of $S(z)$ such that $f(t) = \alpha_t(a)$ for all $t\in\mathbb{R}$. We define $\alpha_z(a) = f(z)$ for $z\in S(z)$. Analytic in this context means that the composition $\varphi\circ f$ is holomorphic on the interior of $S(z)$ for all $\varphi\in\mathcal{A}^*$. We say that an element $a\in\mathcal{A}$ is \emph{analytic for $\alpha$} if there exists an entire function $f:\mathbb{C}\rightarrow\mathcal{A}$ such that $f(t) = \alpha_t(a)$ for all $t\in\mathbb{R}$.

With the basic concepts about $C^*$-dynamical systems established we turn our attention towards \emph{weights}, \cites{Pedersen1979} is a nice resource for learning the basics about weights. Let $\mathcal{A}$ be a $C^*$-algebra and denote by $\mathcal{A}_+$ the convex cone of positive elements in $\mathcal{A}$. A \emph{weight} on $\mathcal{A}$ is a map $\psi:\mathcal{A}_+\rightarrow[0,\infty]$ such that for all $a,b\in\mathcal{A}_+$ and $\lambda\geq0$ we have that
\begin{itemize}
    \item $\psi(a + b) = \psi(a) + \psi(b)$ and
    \item $\psi(\lambda a) = \lambda\psi(a)$.
\end{itemize}
Denote by $\mathcal{A}_+^\psi$ the set of all positive elements $a\in\mathcal{A}_+$ for which $\psi(a)<\infty$. Since the positive elements of a $C^*$-algebra span the entire $C^*$-algebra we define $\mathcal{A}^\psi = \spn(\mathcal{A}_+^\psi)$, and we may extend $\psi$ uniquely to a positive functional on $\mathcal{A}^\psi$ in the obvious way. (This is done in detail in ~\cite{Pedersen1979}*{Lemma 5.1.2}).

We call a weight $\psi$
\begin{itemize}
    \item \emph{densely defined} if $\mathcal{A}_+^\psi$ is dense in $\mathcal{A}_+$,
    \item \emph{lower semi-continuous} if $\{a\in\mathcal{A}_+ \ | \ \psi(a)\leq \lambda\} $ is closed for all $\lambda\geq0$, and
    \item \emph{proper} if it is densely defined and lower semi-continuous. \qedhere
\end{itemize}

Given a $C^*$-dynamical system, it is of interest to ask which weights are invariant under this dynamic. The theory of KMS$_\beta$-weights gives an answer to this question and was first introduced by Combes in \cites{Combes1971}. We will however use the following description of KMS$_\beta$-weights, c.f.~\cite{Kustermans1997}*{Definition 2.8}. Let $(\mathcal{A},\mathbb{R},\alpha)$ be a $C^*$-dynamical system and $\beta\in\mathbb{R}$. We call a weight $\psi$ on $\mathcal{A}$ a \emph{KMS$_\beta$-weight for $\alpha$} if it is a proper weight satisfying
\begin{itemize}
    \item $\psi\circ\alpha_t = \psi$ for all $t\in\mathbb{R}$, and
    \item for every $a\in D(\alpha_{-\beta i/2})$ we have that $\psi(a^*a) = \psi(\alpha_{-\beta i/2}(a)\,\alpha_{-\beta i/2}(a)^*)$. 
\end{itemize}
Note that a \emph{KMS$_\beta$-state for $\alpha$} is defined analogously with the additional criteria of being a state.

\subsection{Dynamical systems of groupoid algebras}

Let $\mathcal{G}$ be an étale groupoid, and $c:\mathcal{G}\rightarrow \mathbb{R}$ be a continuous groupoid homomorphism. We get a group homomorphism $\alpha^c:\mathbb{R}\rightarrow\text{Aut}(C^*(\mathcal{G}))$ by extending the following function, defined by
\begin{equation*}\label{eqn:diagonal action real}
    \alpha_t^c(f)(x) = e^{itc(x)}\,f(x),
\end{equation*}
for $t\in\mathbb{R}$, $f\in C_c(\mathcal{G})$ and $x\in\mathcal{G}$. It is also clear that if $z\in\mathbb{C}$ we have that $\alpha_z^c(f)\in C_c(\mathcal{G})$ for all $f\in C_c(\mathcal{G})$, hence every $f\in C_c(\mathcal{G})$ is analytic for $\alpha^c$. When $\mathcal{G}$ is a principal groupoid Renault is able to classify all the KMS$_\beta$-states for $\alpha^c$ in terms of probability measures on the unit space $\mathcal{G}^{(0)}$ that are \emph{quasi-invariant with Radon-Nikodym cocycle $e^{-\beta c}$}, c.f.~\cites{Renault1980}. A regular Borel measure $\mu$ on the unit space $\mathcal{G}^{(0)}$ is \emph{quasi-invariant with Radon-Nikodym cocycle $e^{-\beta c}$} if the pullbacks $r^*\mu$ and $s^*\mu$ are equivalent and $\frac{d(r^*\mu)}{d(s^*\mu)} = e^{-\beta c}$. In \cites{Renault1980} he defines the pullback of a measure using \cref{eqn:old pullback of measure}. However, we want to think of the pullback differently, which we make precise in \cref{sec:measures}. Denote all the quasi-invariant measures with Radon-Nikodym cocycle $e^{-\beta c}$ by $\Delta(e^{-\beta c})$.

For non-principal groupoids, Neshveyev is able to generalize Renault's description, c.f.~\cites{Neshveyev2013}. Combining \cite{Neshveyev2013}*{Theorem 1.3} with his remark that if $\beta \neq 0$ and $\mu \in \Delta(e^{-\beta c})$, then $\mathcal{G}_u^u \subset c^{-1}(0)$ for $\mu$-a.e. $u\in\mathcal{G}^{(0)}$, we get the following result, which we will need in this paper.
\begin{proposition}{\cites{Neshveyev2013}}\label{prop:Neshveyev's description}$ $\newline
    Let $\mathcal{G}$ be an étale groupoid, $c:\mathcal{G}\rightarrow\mathbb{R}$ be a continuous groupoid homomorphism and $\beta\in\mathbb{R}\setminus\{0\}$. If $c^{-1}(0) \cap \mathcal{G}_u^u = \{u\}$ for all $u\in\mathcal{G}^{(0)}$ we have that there is an affine bijection between the set of KMS$_\beta$-states for $\alpha^c$ on $C^*(\mathcal{G})$ and the set of quasi-invariant probability measures on $\mathcal{G}^{(0)}$ with Radon-Nikodym cocycle $e^{-\beta c}$. \qed
\end{proposition}
\begin{remark}\label{remark:Christensen's generalization of Neshveyev's thm}$ $\newline
    In \cites{Christensen2023} Christensen is able to further build on Neshveyev's work and attains a similar description, only of KMS$_\beta$-\emph{weights} instead of states. Hence, \cref{prop:Neshveyev's description} remains true in the more general case of KMS$_\beta$-weights if we remove the criteria of the measures being probability measures.
\end{remark}

\subsection{Topological graphs}\label{sec:topological graphs}

A lot of the theory for $C^*$-algebras associated to topological graphs was first developed by Katsura in \cites{Katsura2003}. Katsura associated a $C^*$-algebra to the quadruple $E = (E^0, E^1, r, s)$ where $E^0$ is thought of as the vertex space of the graph, $E^1$ is thought of as the edge space of the graph and $r: E^1 \rightarrow E^0$ and $s: E^1 \rightarrow E^0$ are maps indicating the direction of each edge, so we may think of $E$ as a generalization of directed graphs. To be precise we define a \emph{topological graph} $E$ to be a quadruple $(E^0,E^1,r,s)$ where $E^0$ and $E^1$ are locally compact Hausdorff spaces, $s:E^1\rightarrow E^0$ is a local homeomorphism, called the \emph{source map}, and $r:E^1\rightarrow E^0$ is a continuous map called the \emph{range map}. We call $E^0$ and $E^1$ the \emph{vertex} space and \emph{edge} space respectively. In the case where both $E^0$ and $E^1$ are second-countable we will say that the topological graph $E$ is second-countable.

It was shown by Yeend, c.f.~\cites{Yeend2007, Yeend2006}, that the $C^*$-algebra constructed by Katsura admits a groupoid model. Yeend's description is for higher rank topological graphs, for rank 1 topological graphs there exists another equivalent description due to Kumijan and Li, c.f.~\cites{Kumjian2017}. It is their description that we will be using.

For paths in topological graphs we use the following convention: if we have edges $\{a_1,a_2,\dots,a_n\}\subset E^1$ such that $s(a_k) = r(e_{k+1})$ we get a path of length $n$ by composing: $a = a_1a_2\cdots a_n$, we denote its length by $|a| = n$. One can similarly define infinite paths. We then define the set $E^n$ to be the set of all paths $a$ such that $|a| = n$, we similarly define $E^\infty$ to be the set of all paths of infinite length. We can give $E^n$ the structure of a locally compact Hausdorff space by considering it as a subspace of $\prod_{k=1}^{n} E^1$. We can similarly give $E^\infty$ the structure of a Hausdorff space by considering it as a subspace of $\prod_{k=1}^{\infty}E^1$, note that $E^\infty$ may not be locally compact.

We define the \emph{finite path space} $E^* = \bigsqcup_{n=0}^{\infty} E^n$ and endow it with the disjoint union topology making it a locally compact Hausdorff space. We can extend the range and source maps to all of $E^*$ by defining
\begin{equation*}
    r(a) = r(a_1), \ s(a) = s(a_{|a|})
\end{equation*}
whenever $|a|\geq1$. For vertices $v\in E^0$ we define $r(v) = v = s(v)$. We can similarly extend the range map to $E^\infty$ as well. Note that the source of an infinite path is not well-defined.

We would also like to capture the vertices that have edges pointing towards them, but not too many edges. We call these vertices for \emph{regular}. In the discrete graph case the regular vertices are defined as the vertices $v$ such that the inverse image $r^{-1}(v)$ is non-empty and finite. For a topological graph $E$ we say that a vertex $v\in E^0$ is \emph{regular} if there exists a relatively compact open neighborhood $U\subset E^0$ about $v$ such that the inverse image $r^{-1}(\overline{U})$ is compact and $r(r^{-1}(U)) = U$. We denote the set of regular vertices as $E^0_{\text{reg}}$. We call the vertices that are not regular for \emph{singular} and denote them as $E^0_{\text{sng}} = E^0 \setminus E^0_{\text{reg}}$. Finally, we define a finite path to be \emph{singular} if the source of the path is a singular vertex. We denote the set of singular finite paths as $E^*_{\text{sng}} = E^* \cap s^{-1}(E^{0}_{\text{sng}})$.

\subsection{The boundary path space and graph groupoid}

By \cite{Kumjian2017}*{Proposition 4.6} Yeend's \emph{boundary path space} can be described as the disjoint union
\begin{equation*}
    \partial E = E^*_{\text{sng}} \sqcup E^\infty,
\end{equation*}
where it is endowed with a suitable topology. Before we say what the topology on the boundary path space is it will be useful to introduce some notation.

\begin{definition}\label{def:initial_path_segment}$ $\newline
    Let $E$ be a topological graph and $a\in E^*\sqcup E^\infty$. Write $a=a_1a_2\cdots a_{|a|}$. For $k,n\in\mathbb{N}$ such that $k\leq n\leq|a|$ we will write
    \begin{align*}
        a(n) & = a_1a_2 \cdots a_n, \\
        a(k,n) & = a_ka_{k+1} \cdots a_n, \\
        a(0) & = r(a).
    \end{align*}
    If we want to refer to a specific edge in the path we will simply use the subscript notation $a_n$ to denote the $n$-th edge of $a$.
\end{definition}

\begin{lemma}\label{lemma:cont of (n)}$ $\newline
    Let $E$ be a topological graph and $a\in E^*\sqcup E^\infty$. Then the map $a \mapsto a(n)$ is continuous for every $0 \leq n \leq |a|$.
\end{lemma}
\begin{proof}$ $\newline
    If $a \in E^*$, suppose $k \geq n$ and let
    \begin{equation*}
        p^k_n:\prod_{i=1}^{k} E^1 \rightarrow \prod_{i=1}^{n} E^1
    \end{equation*}
    be the projection onto the first $n$ entries. By universality of the product topology, $p^k_n$ is continuous and restricts to a continuous map
    \begin{equation*}
        q^k_n:E^k\rightarrow E^n.
    \end{equation*}
    Then the map $a \mapsto a(n)$ is simply
    \begin{equation*}
        \bigsqcup_{k=n}^{\infty}q^k_n:\bigsqcup_{k=n}^{\infty} E^k \rightarrow E^n,
    \end{equation*}
    which is continuous by universality of the disjoint union topology.

    If $a \in E^\infty$ the map $a \mapsto a(n)$ is simply the restriction of the projection
    \begin{equation*}
        p_n:\prod_{i=1}^{\infty} E^1 \rightarrow \prod_{i=1}^{n} E^1
    \end{equation*}
    to $E^\infty$, which is continuous by universality of the product topology.
\end{proof}

For subsets $S\subset E^*$ we define
\begin{equation*}
    Z(S) = \{ a \in \partial E \ | \ a(n) \in S \text{ for some } n \text{ satisfying } 0 \leq n \leq |a| \}.
\end{equation*}
It is these sets that allows us to endow the boundary path space with a locally compact Hausdorff topology: Let $E$ be a topological graph. The collection
\begin{equation*}\label{eqn:boundary path space base}
    \mathcal{B} = \{ Z(U) \setminus Z(K) \ | \ U \text{ is open in } E^* \text{ and } K \text{ is compact in } E^* \}
\end{equation*}
form a base for a locally compact Hausdorff topology on the boundary path space $\partial E$. See for example \cite{Schafhauser2018}*{Proposition 3.2} for a proof of this fact. We want to extract the following result from Schafhauser's proof, as it will be useful later.

\begin{lemma}{\cites{Schafhauser2018}}\label{lemma:topgraph3}$ $\newline
    Let $E$ be a topological graph and $K\subset E^*$ be compact. Then the set $Z(K)$ is compact in $\partial E$. \qed
\end{lemma}

It is clear that if $E$ is a second-countable topological graph, then the boundary path space with this topology is also second-countable.

To construct a groupoid from the boundary path space we need an action from the natural numbers, including zero. This is done via the \emph{backwards shift map} defined as follows. Let $E$ be a topological graph and $a\in\partial E \setminus E^0$. We define a map $\sigma: \partial E \setminus E^0 \rightarrow \partial E$ by
\begin{align*}
    \sigma(a) & = a(2,|a|), \text{ if } |a| \geq 2 \\
    \sigma(a) & = s(a), \text{ if } |a| = 1.
\end{align*}
By \cite{Schafhauser2018}*{Proposition 3.5}, the backwards shift map is a local homeomorphism.

The graph groupoid is then defined as
\begin{equation*}
    \mathcal{G}_E = \{ (a, k - l, b) \in \partial E \times \mathbb{Z} \times \partial E \ | \ k,l \geq 0, \ k \leq |a|, \ l \leq |b|, \ \sigma^k(a) = \sigma^l(b) \}.
\end{equation*}
For open subsets $U, V \subset \partial E$ and non-negative integers $k,l\in\mathbb{N}_0$ we define the set $Z(U,k,l,V)$ to be
\begin{equation*}
    Z(U,k,l,V) = \{ (a, k - l, b) \in \mathcal{G}_E \ | \ a \in U, \ b \in V \}.
\end{equation*}
It is well known that the collection of these sets form a base for a locally compact Hausdorff étale topology on $\mathcal{G}_E$ since, you might also recognize $\mathcal{G}_E$ as the Deaconu-Renault groupoid associated to the action $\mathbb{N}_0\times\partial E \rightarrow \partial E$ sending $(k,a)\mapsto \sigma^k(a)$. It is clear that if $E$ is a second-countable topological graph then $\mathcal{G}_E$ becomes second-countable as well (hence, by our naming convention, $\mathcal{G}_E$ is an étale groupoid whenever $E$ is a second-countable topological graph), hence we may form the associated groupoid $C^*$-algebra, $C^*(E) := C^*(\mathcal{G}_E)$.

It should also be clear that the unit space of $\mathcal{G}_E$ is homeomorphic to the boundary path space $\partial E$, hence we are ultimately interested in studying measures on $\partial E$.

\subsection{The gauge action}

We get an action from the circle group on the $C^*$-algebra of a second-countable topological graph $E$ by dualizing the map $\Phi : \mathcal{G}_E \rightarrow \mathbb{Z}$ defined by $\Phi(a,n,b) = n$. Explicitly, for $z\in\mathbb{T}$ we get a continuous *-homomorphism $\gamma_z : C_c(\mathcal{G}_E) \rightarrow C_c(\mathcal{G}_E)$ defined by
\begin{equation*}
    \gamma_z(f)(a,n,b) = z^n\,f(a,n,b)
\end{equation*}
for $f\in C_c(\mathcal{G}_E)$ and $(a,n,b)\in\mathcal{G}_E$. By continuity this extends to give us a unique group homomorphism $\gamma: \mathbb{T} \rightarrow \text{Aut}(C^*(E))$ defined by $z \mapsto \gamma_z$. We get a $C^*$-dynamical system, $(C^*(E),\mathbb{R},\alpha^\Phi)$, where $\alpha^\Phi$ is defined by precomposing $\gamma$ with the exponential map, i.e. $\alpha^\Phi_t = \gamma_{e^{it}}$.

\section{Sheaves of measures}\label{sec:measures}

Instead of trying to immediately describe the pullback of a \emph{regular} Borel measure let us take a step back and only consider Borel measures. Let $f:X\rightarrow Y$ be a local homeomorphism between topological spaces and $\mu$ be a Borel measure on $Y$. If $U$ is an open subset of $X$ such that $f|_U$ is injective we have that $U\cong f(U)$, and we want the Borel subsets of $U$ to have the same measure as their image in $f(U)$. So we define the pullback $f|_U^*\mu$ of $\mu$ on $U$ by setting $f|_U^*\mu(B) = \mu(f(B))$ for each Borel subset $B$ in $U$. If we let 
\begin{equation*}
    \mathcal{U} = \{ U \subset X \ | \ U \text{ is open in } X \text{ and } f|_U \text{ is injective} \}
\end{equation*}
we get a family of Borel measures $\{f|_U^*\mu\}_{U\in\mathcal{U}}$. We would like to be able to glue these measures to obtain a global measure on $X$, and we would like the resulting measure to be independent of our choice of open cover of $X$.

Sheaf-theory is precisely the tool that axiomatizes such problems and gives us a clear road map for things we need to prove in order to define the pullback of a measure. To our knowledge using sheaf-theory to study measures in this context is all original work.

\begin{definition}\label{def:presheaf of measures}$ $\newline
    Let $X$ be a topological space. For each open subset $U\subset X$, let
    \begin{equation*}
        \mathcal{M}(U) = \{ \mu:\mathcal{B}(U)\rightarrow[0,\infty] \ | \ \mu \text{ is a Borel measure} \},
    \end{equation*}
where $\mathcal{B}(U)$ denotes the Borel $\sigma$-algebra of $U$. For inclusion of open subsets $V\subset U$ in $X$, we define the restriction map $\rho^U_V:\mathcal{M}(U) \rightarrow \mathcal{M}(V)$ by $\mu|_V(B) = \mu(B)$ where $\mu\in\mathcal{M}(U)$ and $B\subset V$ is a Borel subset. This is well-defined since any Borel subset in $V$ is automatically a Borel subset in $U$, since $V$ is open in $U$. We denote the collection of this data, namely the sets $\mathcal{M}(U)$ and the restriction maps $\rho^U_V:\mathcal{M}(U)\rightarrow\mathcal{M}(V)$, by $\mathcal{M}$.
\end{definition}

\begin{lemma}\label{lemma:presheaf of measures}$ $\newline
    Let $X$ be a topological space, then $\mathcal{M}$ is a presheaf of sets on $X$.
\end{lemma}
\begin{proof}$ $\newline
    Fix an open subset $U\subset X$ and let $\mu\in\mathcal{M}(U)$. Clearly $\mu|_U=\mu$, so $\rho^U_U = id_{\mathcal{M}(U)}$. If we have inclusions of open subsets $W\subset V\subset U$ in $X$, we have for all Borel subsets $B\subset W$ that
    \begin{equation*}
        (\mu|_V)|_W(B) = \mu|_V(B) = \mu(B) = \mu|_W(B).
    \end{equation*}
    Hence, $\rho^V_W\circ\rho^U_V = \rho^U_W$. This completes the proof showing that $\mathcal{M}$ is a contravariant functor from the topology of any topological space to the category of sets.
\end{proof}

For general topological spaces $X$ we won't get that $\mathcal{M}$ is a sheaf. To prove the locality axiom we require that any open cover $\mathcal{U}$ of $X$ has a countable subcover, which isn't guaranteed. The need for this comes from the additivity axiom for measures. We do however get a sheaf for second-countable spaces.

\begin{proposition}\label{prop:sheaf of measures}$ $\newline
    Let $X$ be a second-countable space, then $\mathcal{M}$ is a sheaf of sets on $X$.
\end{proposition}
\begin{proof}$ $\newline
    Note first that second-countability of $X$ implies that any subspace of $X$ is also second-countable. Hence, it suffices to prove that locality and gluing holds for $X$ to conclude that they hold for any open subset of $X$.

    Let $\mathcal{U}$ be an open cover of $X$. Since $X$ is second-countable we get a countable sub-cover $\mathcal{V}$ of $\mathcal{U}$. Let $\mathcal{V} = \{V_i\}_{i\in\mathbb{N}}$ be an indexing of $\mathcal{V}$.

    For subsets $S\subset X$ we introduce the notation $S(i) = S\cap(V_i\setminus(V_1\cup...\cup V_{i-1}))$. Clearly $S = \bigcup_{i\in\mathbb{N}}S(i)$ with $S(i) \cap S(j) = \emptyset$ for $i\neq j$ and each $S(i) \subset V_i$.
    
    \noindent
    \emph{Locality:} \\
    Suppose we have Borel measures $\mu,\nu\in\mathcal{M}(X)$ such that $\mu|_U = \nu|_U$ for all $U\in\mathcal{U}$. Then for any Borel subset $B\subset X$ we have that
    \begin{equation*}
        \mu(B) 
        = \mu\left(\bigcup_{i\in\mathbb{N}}B(i)\right)
        = \sum_{i\in\mathbb{N}}\mu(B(i))
        = \sum_{i\in\mathbb{N}}\mu|_{V_i}(B(i))
        = \sum_{i\in\mathbb{N}}\nu|_{V_i}(B(i))
        = \nu(B).
    \end{equation*}
    \noindent
    \emph{Gluing:} \\
    Suppose we have a Borel measures $\mu_U\in\mathcal{M}(U)$ for each $U\in\mathcal{U}$ such that $\mu_U|_{U\cap V} = \mu_V|_{U\cap V}$ for all $U,V \in \mathcal{U}$. Define $\mu_i = \mu_{V_i}$. We define a global Borel measure $\mu\in\mathcal{M}(X)$ by
    \begin{equation*}
        \mu(B) = \sum_{i\in\mathbb{N}}\mu_i(B(i)),
    \end{equation*}
    where $B\subset X$ is a Borel subset. The only non-trivial thing we need to check to ensure that $\mu$ is a Borel measure is additivity:
    \begin{align*}
        \mu\left( \bigcup_{j\in\mathbb{N}}B_j \right)
        & = \sum_{i\in\mathbb{N}}\mu_i\left( \bigcup_{j\in\mathbb{N}}B_j(i) \right)
        = \sum_{i\in\mathbb{N}}\sum_{j\in\mathbb{N}}\mu_i(B_j(i))
        = \sum_{j\in\mathbb{N}}\sum_{i\in\mathbb{N}}\mu_i(B_j(i))
        = \sum_{j\in\mathbb{N}}\mu(B_j),
    \end{align*}
    where $\{B_j\}_{j\in\mathbb{N}}$ is a collection of mutually disjoint Borel subsets in $X$. Note that we can interchange the sums in the last line because we are summing positive numbers.
    
    Now let $U\in\mathcal{U}$ and $B\subset U$ be a Borel subset. To see that $\mu|_U = \mu_U$ is a matter of unraveling definitions and noting that each $B(i)\subset U\cap V_i$, hence $\mu_i|_{U\cap V_i}(B(i)) = \mu_U|_{U\cap V_i}(B(i))$.
\end{proof}

Note that the sets $\mathcal{M}(U)$, \emph{almost} has the structure of a real vector space, since we can add measures and scale them by non-negative real numbers. One could consider the group completion of $\mathcal{M}(U)$ by adding formal inverses to different measures, i.e. by adding \emph{virtual measures}. If we do this we get that $\mathcal{M}$ takes values in $\mathbf{Ab}$, so for second-countable spaces, $\mathcal{M}$ would be a sheaf of abelian groups (in fact a sheaf of real vector spaces). We do not explore this further in this paper, but it is worth noting.

We now give an example of a (not second-countable) space for which locality does not hold.

\begin{example}\label{example:counter example sheaf}$ $\newline
    Consider $\mathbb{R}$ with the discrete topology. Let $\mathcal{U} = \{\{x\}\}_{x\in\mathbb{R}}$ be an open covering of $\mathbb{R}$. Let $\mu$ be the measure defined by
    \begin{equation*}
        \mu(B) = \left\{
        \begin{matrix*}[l]
            0 & \text{if } B \text{ is countable or finite},\\
            \infty & \text{if } B \text{ is uncountable}.
        \end{matrix*}
        \right.
    \end{equation*}
    This clearly defines a measure. It is clear that $\mu$ agrees with the zero measure on all restrictions to $\{x\}\in\mathcal{U}$, however $\mu$ is not the zero measure, so locality doesn't hold.
\end{example}

We are however mostly interested in \emph{regular} Borel measures. Some authors give different definitions of a regular measure, we will say that a Borel measure is regular if it is both inner- and outer-regular \emph{and} finite on compact sets. Recall that for locally compact Hausdorff spaces we have that a Borel measure is regular if and only if it is finite on compact sets.

We now want to find the requirements on a topological space for which $\mathcal{M}_{\text{reg}}$, the family of \emph{regular} Borel measures, becomes a sheaf of sets.

\begin{lemma}\label{lemma:presheaf of regular measures}$ $\newline
    Let $X$ be a topological space, then $\mathcal{M}_{\text{reg}}$ is a sub-presheaf of $\mathcal{M}$ on $X$.
\end{lemma}
\begin{proof}$ $\newline
    We clearly have that $\mathcal{M}_{\text{reg}}(U) \subset \mathcal{M}(U)$ for all open subset $U \subset X$. We want to define the restriction of $\mathcal{M}_{\text{reg}}$ to be the restriction of $\mathcal{M}$, i.e. that it is given by the composition
    \begin{equation*}
        \xymatrix{
            \mathcal{M}_{\text{reg}}(V) \ar@{^{(}->}[r] & \mathcal{M}(V) \ar[r] & \mathcal{M}(U)
        }
    \end{equation*}
    for inclusions of open subsets $U \subset V \subset X$. We want to check that this composition factors through $\mathcal{M}_{\text{reg}}(U)$, then we immediately get that $\mathcal{M}_{\text{reg}}$ is a sub-presheaf of $\mathcal{M}$ on $X$. To do this we only need to check that the restriction of a regular Borel measure (viewed as a section of $\mathcal{M}(V)$) is again a regular Borel measure. Doing this is standard procedure, so we omit the details here.
\end{proof}

\begin{proposition}\label{prop:sheaf of regular measures}$ $\newline
    Let $X$ be a second-countable locally compact Hausdorff space. Then $\mathcal{M}_{\text{reg}}$ is a sub-sheaf of $\mathcal{M}$ on $X$.
\end{proposition}
\begin{proof}$ $\newline
    Note that since $\mathcal{M}_{\text{reg}}$ is a sub-presheaf of $\mathcal{M}$ we only need to check that $\mathcal{M}_{\text{reg}}$ is a sheaf of sets on $X$ to conclude that $\mathcal{M}_{\text{reg}}$ is a sub-sheaf of $\mathcal{M}$.

    We have that any subspace of $X$ is second-countable locally compact and Hausdorff, hence it suffices to prove that locality and gluing holds for $X$ to conclude that it holds for any open subset of $X$.

    Let $\mathcal{U}$ be an open cover of $X$. Locality for $\mathcal{M}_{\text{reg}}$ follows by the same argument as for $\mathcal{M}$. To prove gluing, suppose we have regular Borel measures $\mu_U\in\mathcal{M}_{\text{reg}}(U)$ for each $U\in\mathcal{U}$ such that $\mu_U|_{U\cap V} = \mu_V|_{U\cap V}$ for all $U,V \in \mathcal{U}$. We get a global measure $\mu\in\mathcal{M}(X)$ that restricts to each $\mu_U$ for $U\in\mathcal{U}$ by \cref{prop:sheaf of measures}. It remains to check that $\mu$ is in fact regular. Let
    \begin{equation*}
        \mathcal{V} = \{ V\subset X \ | \ V \text{ is open in } X \text{ and } \overline{V}\subset U \text{ for some } U\in\mathcal{U} \}.
    \end{equation*}
    This is an open cover of $X$: For any $x\in X$ let $U\in\mathcal{U}$ be a neighborhood about $x$. Then $X\setminus U$ is closed, and since $X$ is locally compact Hausdorff, $X$ is in particular regular, so we can find disjoint open subsets $V$ and $W$ containing $x$ and $X\setminus U$ respectively. Then $\overline{V}\subset U$, showing that $V\in\mathcal{V}$.

    Let now $K\subset X$ be compact, then $\mathcal{V}$ is an open cover of $K$, so it admits a finite sub-cover $\{V_i\}_{i=1}^{n}$ of $K$. Let $U_i\in\mathcal{U}$ be such that $V_i\subset\overline{V_i}\subset U_i$. Then $K\cap \overline{V_i} \subset U_i$ is compact. Thus,
    \begin{equation*}
        \mu(K) 
        \leq \sum_{i=1}^{n}\mu(K\cap \overline{V_i})
        = \sum_{i=1}^{n}\mu|_{U_i}(K\cap \overline{V_i})
        = \sum_{i=1}^{n}\mu_{U_i}(K\cap \overline{V_i})
        < \infty.
    \end{equation*}
    Hence, since $X$ is locally compact Hausdorff we get that $\mu$ is in fact regular, proving that gluing holds for $\mathcal{M}_{\text{reg}}$.
\end{proof}

With this in place we can easily define the pullback of a regular Borel measure.

\begin{proposition}\label{prop:pullback of measure}$ $\newline
    Let $f:X \rightarrow Y$ be a local homeomorphism between second-countable locally compact Hausdorff spaces and $\mu$ be a regular Borel measure on $Y$. Then there exists a unique regular Borel measure $f^*\mu$ on $X$, the \emph{pullback of $\mu$}, such that $f^*\mu(U) = \mu(f(U))$ for all open subsets $U\subset X$ where $f|_U$ is injective.
\end{proposition}
\begin{proof}$ $\newline
    Let
    \begin{equation*}
        \mathcal{U} = \{U\subset X \ | \ U \text{ is open in } X \text{ and } f|_U \text{ is injective} \}.
    \end{equation*}
    Since $f$ is a local homeomorphism this is an open cover of $X$. For each $U\in\mathcal{U}$ we get a regular Borel measure $f^*\mu_U\in\mathcal{M}_{\text{reg}}(U)$ by the equation $f^*\mu_U(B) = \mu(f(B))$, where $B\subset U$ is a Borel subset. The fact that this defines a Borel measure is obvious. Regularity follows as such: Let $K\subset U$ be compact, then $f^*\mu_U(K) = \mu(f(K)) < \infty$ by compactness of $f(K)$ which follows by continuity of $f$.

    Now we need to show that these measures agree on intersections. Let $U,V\in\mathcal{U}$. Then for any Borel subset $B\subset U\cap V$ we have that
    \begin{equation*}
        f^*\mu_U|_{U\cap V}(B) = f^*\mu_U(B) = \mu(f(B)) = f^*\mu_V|_{U\cap V}(B).
    \end{equation*}
    Since $\mathcal{M}_{\text{reg}}$ is a sheaf on $X$ by \cref{prop:sheaf of regular measures}, we get that there exists a unique regular Borel measure $f^*\mu$ on $X$ such that for any $U\in\mathcal{U}$
    \begin{equation*}
        f^*\mu(U) = f^*\mu|_U(U) = f^*\mu_U(U) = \mu(f(U)). \qedhere
    \end{equation*}
\end{proof}

We are now also able to prove that this definition of the pullback agrees with the one obtained by \cref{eqn:old pullback of measure}. Even though the pullback as defined through \cref{eqn:old pullback of measure} is a well-known construction, we choose to present the construction here. We will omit the proof of this next result as it is easily obtained using standard methods.

\begin{lemma}\label{lemma:pullback of measure 2}$ $\newline
    Let $X$ and $Y$ be topological spaces, $\varphi:X\rightarrow Y$ a local homeomorphism and $f\in C_c(X)$. Then the function
    \begin{equation*}
        y\mapsto \sum_{x\in\varphi^{-1}(y)}f(x)
    \end{equation*}
    from $Y$ to the complex numbers $\mathbb{C}$ is well-defined and continuous with compact support. \qed
\end{lemma}

With this Lemma we get a linear map $\mathcal{L}:C_c(X) \rightarrow C_c(Y)$ defined by
\begin{equation*}
    (\mathcal{L}f)(y) = \sum_{x\in\varphi^{-1}(y)}f(x).
\end{equation*}
Maps of this kind are called \emph{transfer operators}, see for example \cites{Baladi2000} for an introduction to their use in dynamics.

\begin{proposition}\label{prop:pullback of measure 2}$ $\newline
    Let $\varphi:X\rightarrow Y$ be a local homeomorphism between locally compact Hausdorff spaces and $\mu$ be a regular Borel measure on $Y$. Then there exists a unique regular Borel measure $\varphi^*\mu$ on $X$ such that
    \begin{equation*}
        \int_X f d\varphi^*\mu = \int_Y \sum_{x\in\varphi^{-1}(y)}f(x)d\mu(y),
    \end{equation*}
     where $f\in C_c(X)$.
\end{proposition}
\begin{proof}$ $\newline
    The proof follows by \cref{lemma:pullback of measure 2} along with Riesz representation Theorem.
\end{proof}

\begin{proposition}\label{prop:uniqueness of pullback}$ $\newline
    Let $\varphi:X\rightarrow Y$ be a local homeomorphism between second-countable locally compact Hausdorff spaces and $\mu$ be a regular Borel measure on $Y$. Then the pullback of $\mu$ as in \cref{prop:pullback of measure} equals the one in \cref{prop:pullback of measure 2}.
\end{proposition}
\begin{proof}$ $\newline
    Let $\lambda$ be the pullback of $\mu$ as in \cref{prop:pullback of measure} and $\nu$ be the pullback of $\mu$ as in \cref{prop:pullback of measure 2}. Let
    \begin{equation*}
        \mathcal{U} = \{U\subset X \ | \ U \text{ is open in } X \text{ and } \varphi|_U \text{ is injective} \}.
    \end{equation*}
    This is an open cover of $X$ since $\varphi$ is a local homeomorphism. By \cref{prop:sheaf of regular measures} we know that $\mathcal{M}_{\text{reg}}$ is a sheaf on $X$, so we need only check that $\lambda|_U = \nu|_U$ for each $U\in\mathcal{U}$ to conclude that $\lambda = \nu$. Let $U\in\mathcal{U}$ and $B\subset U$ be a Borel subset. Then
    \begin{align*}
        \nu|_U(B)
        & = \nu(B) \\
        & = \int_Y\sum_{x\in\varphi^{-1}(y)}\chi_B(x) d\mu(y) \\
        & = \int_Y\chi_{\varphi(B)} d\mu \\
        & = \lambda|_U(B). \qedhere
    \end{align*}
\end{proof}

We end this section by briefly recalling the definition of the \emph{pushforward} of a Borel measure: Let $X$ and $Y$ be topological spaces, $f:X\rightarrow Y$ be continuous and $\mu$ be a Borel measure on $X$. We get a Borel measure $f_*\mu$ on $Y$, the \emph{pushforward of $\mu$}, defined by $f_*\mu(B) = \mu(f^{-1}(B))$ for Borel subsets $B\subset Y$. It is clear that for the pushforward of regular measure to be regular it is sufficient for the function $f$ in this definition to be proper.

\section{Quasi-invariant measures on a topological graph}\label{sec:quasi-invariant measures}

We want to specialize \cref{prop:Neshveyev's description} for the graph groupoid, $\mathcal{G}_E$, of a second-countable topological graph. This amounts to giving a description of the set $\Delta(e^{-\beta\Phi})$ only using data from the topological graph. Christensen does this for the $\beta = 0$ case, i.e. he obtains a description of the tracial-weights on $C^*(E)$, c.f.~\cite{Christensen2022}*{Theorem 1.3}, note that describing the tracial-weights is a bit more convoluted than describing the KMS$_{\beta}$-weights for values of $\beta\neq0$.

The regular Borel measures $\nu$ on $\partial E$ that will be of interest will turn out to be those that satisfy the following equality: $\sigma^* \nu = e^\beta \nu$ on $\partial E \setminus E^0$. Since we want to mainly think about this property we make the following definition.

\begin{definition}\label{def:quasi inv measures}$ $\newline
    Let $E$ be a second-countable topological graph and $\beta\in\mathbb{R}$. We say that a regular Borel measure $\nu$ on the boundary path space $\partial E$ is \emph{$\beta$-quasi-invariant} if $\sigma^*\nu = e^\beta \nu$ on $\partial E \setminus E^0$. We denote the set of all $\beta$-quasi invariant measures by $\mathcal{M}^\beta_{\text{quasi}}(E)$.
\end{definition}

In the next result the range and source maps, denoted $r$ and $s$ respectively, will refer to the range and source maps of the groupoid $\mathcal{G}_E$, not the topological graph.

\begin{lemma}\label{lemma:invariant measures graph}$ $\newline
    Let $E$ be a second-countable topological graph, $\beta\in\mathbb{R}\setminus\{0\}$, $\Phi:\mathcal{G}_E\rightarrow\mathbb{Z}$ be given by $\Phi(a,n,b) = n$. Then $\Delta(e^{-\beta\Phi}) \cong \mathcal{M}^\beta_{\text{quasi}}(E)$.
\end{lemma}

\begin{proof}$ $\newline
    We first want an easier way to check if a regular Borel measure is quasi-invariant with Radon-Nikodym cocycle $e^{-\beta\Phi}$. To that end, suppose that $\nu \in \Delta(e^{-\beta\Phi})$ and let $B \subset \mathcal{G}_E$ be a Borel subset. Then we have that $B \cap \Phi^{-1}(\{k\})$ is a Borel subset for any $k \in \mathbb{Z}$ since $\Phi$ is continuous. Hence, 
    \begin{align*}
        r^*\nu(B)
        & = \int_B e^{-\beta\Phi}d(s^*\nu) \\
        & = \sum_{k\in\mathbb{Z}} \int_{B\cap\Phi^{-1}(\{k\})} e^{-\beta k}d(s^*\nu) \\
        & = \sum_{k\in\mathbb{Z}} e^{-\beta k} s^*\nu(B \cap \Phi^{-1}(\{k\})).
    \end{align*}
    So we have that $\nu\in\Delta(e^{-\beta\Phi})$ if and only if
    \begin{equation}\label{eqn:quasi-inv graph}
        r^*\nu = \sum_{k\in\mathbb{Z}} e^{-\beta k} (s^*\nu)|_{\Phi^{-1}(\{k\})},
    \end{equation}
    where the notation $(s^*\nu)|_{\Phi^{-1}(\{k\})}$ simply means that
    \begin{equation*}
        (s^*\nu)|_{\Phi^{-1}(\{k\})}(B) = s^*\nu(B\cap\Phi^{-1}(\{k\}))
    \end{equation*}
    for Borel subsets $B\subset\mathcal{G}_E$. Note that if $W\subset \mathcal{G}_E$ is an open bisection we have that \cref{eqn:quasi-inv graph} takes the form
    \begin{equation*}
        \nu(r(W)) = \sum_{k\in\mathbb{Z}}e^{-\beta k}\nu(s(W\cap\Phi^{-1}(\{k\}))).
    \end{equation*}

    Now consider the collection
    \begin{equation*}
        \mathcal{U} = \{ U\subset\partial E\setminus E^0 \ | \ U \text{ is open in } \partial E \text{ and } \sigma|_U \text{ is injective} \}.
    \end{equation*}
    Since $\sigma$ is a local homeomorphism we have that $\mathcal{U}$ is an open cover of $\partial E\setminus E^0$. By \cref{prop:sheaf of regular measures}, $\mathcal{M}_{\text{reg}}$ is a sheaf on $\partial E\setminus E^0$, so it suffices to check that $\sigma^*\nu = e^\beta\nu$ on each $U\in\mathcal{U}$ to conclude that they are equal on $\partial E\setminus E^0$.
    
    We claim that for any $U\in\mathcal{U}$ the basic open set $Z(U,1,0,\sigma(U)) \subset \mathcal{G}_E$ is a bisection. Indeed, any element in $Z(U,1,0,\sigma(U))$ is of the form $(a,1,\sigma(a))$ by injectivity of $\sigma|_U$. Hence, the equations
    \begin{align*}
        a  
       &= r( a ,1,\sigma( a ))
       = r( b,1,\sigma( b))
       =  b, \\
       \sigma( a ) 
       &= s( a ,1,\sigma( a ))
       = s( b,1,\sigma( b))
       = \sigma( b),
   \end{align*}
   show that $r|_{Z(U,1,0,\sigma(U))}$ and $s|_{Z(U,1,0,\sigma(U))}$ are injective.

   Fix $U\in\mathcal{U}$ and let $B\subset U$ be a Borel subset. It is clear that $Z(B,1,0,\sigma(B))$ is a Borel subset of $Z(U,1,0,\sigma(U))$ and that $Z(B,1,0,\sigma(B))\cap\Phi^{-1}(\{k\})$ is nonempty if and only if $k=1$, and in this case $Z(B,1,0,\sigma(B))\cap\Phi^{-1}(\{1\}) = Z(B,1,0,\sigma(B))$. Hence, by \cref{eqn:quasi-inv graph} we get that
   \begin{align*}
    \nu(B) 
    & = r^*\nu(Z(B,1,0,\sigma(B))) \\
    & = e^{-\beta}s^*\nu(Z(B,1,0,\sigma(B))) \\
    & = e^{-\beta}\nu(\sigma(B)).
   \end{align*}
   So we indeed have that $\sigma^*\nu = e^\beta \nu$ on each $U \in \mathcal{U}$, so we have that $\sigma^*\nu = e^\beta \nu$ on $\partial E \setminus E^0$.
   
   For the converse we will show that \cref{eqn:quasi-inv graph} holds. By \cref{prop:sheaf of regular measures}, $\mathcal{M}_{\text{reg}}$ is a sheaf on $\mathcal{G}_E$, so it suffices to check that \cref{eqn:quasi-inv graph} holds on each open subset in some suitable open cover of $\mathcal{G}_E$. The open cover we will use is the following:
   \begin{equation*}
    \mathcal{V} = \{ Z(U,m,n,V) \ | \ m,n\in\mathbb{N}_0, U\in\mathcal{U}_m \text{ and } V\in\mathcal{U}_n \},
   \end{equation*}
   where
   \begin{equation*}
    \mathcal{U}_k = \{ U\subset \partial E\setminus E^{k-1} \ | \ U \text{ is open in } \partial E\setminus E^{k-1} \text{ such that } \sigma^k|_U \text{ is injective} \}.
   \end{equation*}

   We first show that $\mathcal{V}$ is an open cover of $\mathcal{G}_E$: Let $(a,m-n,b)\in\mathcal{G}_E$. Then $|a|\geq m$ and $|b|\geq n$, so $a\in\partial E \setminus E^{m-1}$ and $b\in \partial E \setminus E^{n-1}$. For every $k\in\mathbb{N}_0$ we have that $\sigma^k$ is a local homeomorphism, hence there exists a $U\in \mathcal{U}_m$ and a $V\in\mathcal{U}_n$ such that $a\in U$ and $b\in V$. Hence, $(a,m-n,b)\in Z(U,m,n,V)$ where $Z(U,m,n,V)\in\mathcal{V}$.

   Next we claim that each $Z(U,m,n,V)\in\mathcal{V}$ is a bisection. Indeed, let $(a_1,m-n,b_1),(a_2,m-n,b_2)\in Z(U,m,n,V)$ and suppose $r(a_1,m-n,b_1) = r(a_2,m-n,b_2)$. By definition of the range map we get that $a_1 = a_2$. Then we get that
   \begin{equation*}
    \sigma^n(b_1) = \sigma^m(a_1) = \sigma^m(a_2) = \sigma^n(b_2).
   \end{equation*}
   By injectivity of $\sigma^n|_V$ we get that $b_1 = b_2$, hence $r|_{Z(U,m,n,V)}$ is injective. Injectivity of $s|_{Z(U,m,n,V)}$ is similarly proven to be true.

   We can now show that any regular Borel measure on $\partial E$ that satisfies $\sigma^*\nu = e^{\beta} \nu$ on $\partial E \setminus E^0$ must also satisfy \cref{eqn:quasi-inv graph}. Fix $Z(U,m,n,V) \in \mathcal{V}$ and let $B\subset Z(U,m,n,V)$ be a Borel subset. It is clear that $B\cap\Phi^{-1}(\{k\})$ is nonempty if and only if $k=m-n$, and in this case $B\cap\Phi^{-1}(\{m-n\}) = B$. Also since $r|_B$ and $s|_B$ are injective, we have that $B = Z(r(B),m,n,s(B))$. By linearity of the pullback we have that $e^{\beta k}\nu = (\sigma^k)^*\nu$ on $\partial E \setminus E^{k-1}$ for every $k\in\mathbb{N}$, hence
   \begin{align*}
    e^{\beta m}\nu(r(B)) 
    & = (\sigma^m)^*\nu(r(B)) \\
    & = \nu(\sigma^m(r(B))) \\
    & = \nu(\sigma^n(s(B))) \\
    & = (\sigma^n)^*\nu(s(B)) \\
    & = e^{\beta n}\nu(r(B)).
   \end{align*}
   Note that equality two and four follows by injectivity of $\sigma^m|_U$ and $\sigma^n|_V$ and equality three also requires the fact that $Z(U,m,n,V)$ is a bisection. Hence, $\nu(r(B)) = e^{-\beta(m-n)}\nu(s(B))$ so \cref{eqn:quasi-inv graph} is satisfied, completing the proof.
\end{proof}

With these results in place we will no longer need to refer to the source and range maps of the graph groupoid, hence from now on, when we write $r$ and $s$ we will always refer to the range and source maps of the topological graph.

The isotropy subgroups of the graph groupoid has a particularly nice interpretation: they measure the period of cycles in the graph. Christensen, in \cites{Christensen2022}, introduces the following definition. Let $E$ be a second-countable topological graph and $a\in\partial E$. The \emph{periodicity group of $a$} is
\begin{equation*}
    Per(a) = \{ k - l \in \mathbb{Z} \ | \ k,l\in\mathbb{N}_0, \ k,l\leq |a| \text{ and } \sigma^k(a) = \sigma^l(a) \}.
\end{equation*}
It is clear that this is isomorphic to the isotropy subgroup $(\mathcal{G}_E)^a_a$. It should also be clear that $\Phi^{-1}(0) \cap (\mathcal{G}_E)^a_a = \{a\}$ for all $a\in\partial E$, hence we are in the situation where \cref{prop:Neshveyev's description} holds. Furthermore, it should be clear that $(\mathcal{G}_E)^a_a \subset \Phi^{-1}(0)$ if and only if $Per(a)$ is trivial.

\begin{proposition}\label{prop:KMS weights graph}$ $\newline
    Let $E$ be a second-countable topological graph and $\beta\in\mathbb{R}\setminus\{0\}$. Then there is an affine bijection between the set of KMS$_\beta$-weights(states) for the gauge-action on $C^*(E)$ and (probability measures in) $\mathcal{M}^\beta_{\text{quasi}}(E)$. If $\nu \in \mathcal{M}^\beta_{\text{quasi}}(E)$ (is a probability measure) we get a KMS$_\beta$-weight (state) for the gauge-action on $C^*(E)$, $\psi_\nu$ by the following equation
    \begin{equation*}
        \psi_\nu(f) = \int_{\partial E}  f(a,0,a) d\nu(a)
    \end{equation*}
    for all $f\in C_c(\mathcal{G}_E)$.
\end{proposition}
\begin{proof}$ $\newline
    By the remark that $\Phi^{-1}(0) \cap Per(a) = \{a\}$ for all $a \in \partial E$ we have that \cref{prop:Neshveyev's description} holds. Hence, by \cref{remark:Christensen's generalization of Neshveyev's thm}, there is an affine bijection between $\Delta(e^{-\beta\Phi})$ and the set of KMS$_\beta$-weights for the gauge-action on $C^*(E)$. By \cref{lemma:invariant measures graph} we have that $\Delta(e^{-\beta\Phi}) \cong \mathcal{M}^\beta_{\text{quasi}}(E)$, completing the proof.
\end{proof}

Since the only measures that give rise to KMS$_\beta$-weights (states) for the gauge action on $C^*(E)$ are the ones that measure the boundary paths with trivial periodicity groups, it would be nice to be able to describe when the periodicity groups are trivial. To do this we need to study the cycles in the topological graph.

For a second-countable topological graph $E$ we say that a path $c \in E^*$ is a \emph{cycle} if $|c| \geq 1$ and $r(c) = s(c)$. We denote the set of all cycles in $E$ by $\Omega E$. Every cycle $c\in\Omega E$ gives rise to a well-defined infinite path $c^\infty = cc\cdots$. We further say that a path $b\in E^*$ is an \emph{exit of $c$} if $s(b) = r(c)$ and $b$ is not of the form $ac$ for any path $a\in E^*$. We denote the set of exits of $c$ by $E^*_c$.

It is also interesting to study the infinite paths which are \emph{eventually cyclic}. To be precise we define a path $a \in E^\infty$ to be eventually cyclic if there exists a cycle $c\in\Omega E$ and an exit $b\in E^*_c$ such that $a = bc^\infty$. We denote the set of all eventually cyclic paths in $E$ by $\mathcal{C}E$.

\begin{proposition}\label{lemma:periodicity of eventually cyclic paths}$ $\newline
    Let $E$ be a topological graph and $a\in\partial E$. Then $Per(a) \neq \{0\}$ if and only if $a\in\mathcal{C}E$.
\end{proposition}
\begin{proof}$ $\newline
    Suppose first that $Per(a) \neq \{0\}$. Then there exists integers $k,l\in\mathbb{N}_0$ with $k > l$ such that $\sigma^k(a) = \sigma^l(a)$. Define $d = \sigma^l(a)$, we want to show that the path segment $d(k - l)$ is a cycle (recall the notation from \cref{def:initial_path_segment}). Indeed,
    \begin{equation*}
        \sigma^{k-l}(d) = \sigma^{k-l}(\sigma^l(a)) = \sigma^k(a) = \sigma^l(a) = d,
    \end{equation*}
    showing that the path segment $d(k-l)$ is a cycle. To avoid cluttered notation we define $c = d(k-l)$ to be this cycle, and we will also define $b = a(l)$. It is then clear that we can write $a = bc^\infty$, and if $k$ and $l$ where chosen minimally we get that $b$ is an exit of $c$, hence we may conclude that $a\in\mathcal{C}E$.

    For the converse statement we now assume that $a\in\mathcal{C}E$. By definition there exists a cycle $c\in\Omega E$ and an exit $b\in E^*_c$ such that $a = bc^\infty$. Set $k = |b| + |c|$ and $l = |b|$. Since any cycle has a strictly positive length we get that $k > l$. It is clear that $\sigma^k(a) = \sigma^l(a)$, hence $k-l \in Per(a)$ showing that $Per(a) \neq \{0\}$. In fact $Per(a) = |c|\mathbb{Z}$.
\end{proof}

This result gives an easy proof of the following interesting fact about the graph groupoid.

\begin{corollary}\label{cor:graph groupoid principal}$ $\newline
    Let $E$ be a topological graph. Then the graph groupoid $\mathcal{G}_E$ is principal if and only if $E$ does not contain any cycles.
\end{corollary}
\begin{proof}$ $\newline
    This follows by \cref{lemma:periodicity of eventually cyclic paths} and the observation that the periodicity groups are isomorphic to the isotropy subgroups of the graph groupoid.
\end{proof}

\section{Sub-invariant measures on a topological graph}\label{sec:sub-inv measures}

If our goal is to understand KMS$_\beta$-weights for the gauge action on the $C^*$-algebra of a second-countable topological graph, \cref{prop:KMS weights graph} might not be the best tool, since it will in general be quite difficult to get a nice description of the boundary path space. In this section we show that there is an affine bijection between the $\beta$-quasi-invariant measures on the boundary path space and \emph{$\beta$-sub-invariant measures} on the vertex space of a topological graph. These latter mentioned measures will in general be easier to get a concrete grasp of.

Recall that the range map of a topological graph $E$ naturally extends to give us a map $r:\partial E \rightarrow E^0$. By definition of the topology on $\partial E$ and \cref{lemma:topgraph3} it is clear that this map is not only continuous but also proper. Thus, any regular Borel measure on $\partial E$ can be pushforwarded to give a regular Borel measure on $E^0$. It is therefore of interest to study the regular Borel measures on the vertex space $E^0$ which correspond to the KMS$_\beta$-weights for the gauge-action on $C^*(E)$. Towards that endeavor we make the following definition.

\begin{definition}\label{def:sub-inv-measures}$ $\newline
    Let $E$ be a second-countable topological graph and $\beta \in \mathbb{R}$. We define the map $T:\mathcal{M}_{\text{reg}}(E^0)\rightarrow\mathcal{M}(E^0)$ by $T = r_*s^*$ where $r_*$ and $s^*$ are the pushforward and pullback induced by the range and source maps of the topological graph. We say that a regular Borel measure $\mu$ on the vertex space $E^0$ is \emph{$\beta$-sub-invariant} if $T\mu \leq e^\beta \mu$ on $E^0$ with equality on $ E^0_{\text{reg}}$. We denote the set of all $\beta$-sub-invariant measures by $\mathcal{M}_{\text{sub}}^\beta(E)$.
\end{definition}

Note that setting $\beta = 0$ in this definition recovers similar definitions by Christensen and Schafhauser used to study tracial-weights and -states, c.f.~\cites{Christensen2022,Schafhauser2018}. In fact, we want to check that some of their methods, mainly those by Schafhauser, generalizes to the study of KMS$_\beta$-weights.

The operator $T$ can be thought of as flow in the graph. For a very simple graph,
\begin{equation*}
    \xymatrix{
        \bullet_{u} & \bullet_{v} \ar[l]^e
    }
\end{equation*}
one computes easily that $T\delta_v = \delta_{u}$. Or for a slightly more sophisticated example,
\begin{equation*}
    \xymatrix{
        \bullet_{u} & \bullet_{v} \ar[l]^e \ar[r]_f & \bullet_{w}
    }
\end{equation*}
we have that $T\delta_v = \delta_u + \delta_w$.

\begin{lemma}\label{lemma:range induces quasi to sub}$ $\newline
    Let $E$ be a second-countable topological graph and $\beta\in\mathbb{R}$. The pushforward $r_*:\mathcal{M}^\beta_{\text{quasi}}(E)\rightarrow\mathcal{M}^\beta_{\text{sub}}(E)$ induced by the range map $r:\partial E \rightarrow E^0$ is well-defined.
\end{lemma}
\begin{proof}$ $\newline
    Since the range map is proper by our discussion at the beginning of this section we get that the pushforward of a $\beta$-quasi-invariant measure $\nu$ is at least regular. Hence, we only need to check that $T(r_*\nu) \leq e^\beta (r_*\nu)$ on $E^0$ with equality on $ E^0_{\text{reg}}$.
    
    Let $\mathcal{U} = \{ U\subset E^1 \ | \ U \text{ is open in } E^1 \text{ and } s|_U \text{ is injective} \}$. Since the source map is a local homeomorphism, $\mathcal{U}$ is an open cover of $E^1$. Since $E^1$ is second-countable, $\mathcal{U}$ admits a countable sub-cover $\{U_i\}_{i\in\mathbb{N}}$. Fix a Borel subset $B\subset E^0$. Then
    \begin{align*}
        T(r_*\nu)(B)
        & = \sum_{i\in\mathbb{N}} r_*\nu(s(r^{-1}(B)\cap(U_i\setminus(U_1\cup...\cup U_{i-1})))) \\
        & = \sum_{i\in\mathbb{N}} \nu(Z(s(r^{-1}(B)\cap(U_i\setminus(U_1\cup...\cup U_{i-1}))))) \\
        & = \sum_{i\in\mathbb{N}} \nu(\sigma(Z(r^{-1}(B)\cap(U_i\setminus(U_1\cup...\cup U_{i-1}))))) \\
        & = e^{\beta} \sum_{i\in\mathbb{N}} \nu(Z(r^{-1}(B)\cap(U_i\setminus(U_1\cup...\cup U_{i-1})))) \\
        & = e^{\beta} \nu(Z(r^{-1}(B))) \\
        & \leq e^{\beta} \nu(Z(B)) \\
        & = e^{\beta} (r_*\nu)(B),
    \end{align*}
    where the inequality in the penultimate line is an equality if $B\subset E^0_{\text{reg}}$. Note that the transition from line three to line four follows by the definition of $\mathcal{M}^\beta_{\text{quasi}}(E)$ and the equation for the pullback given in \cref{prop:pullback of measure}. Thus, $r_*\nu \in \mathcal{M}^\beta_{\text{sub}}(E^0)$.
\end{proof}

To prove that the pushforward is injective we will use the fact that if we want to check if two measures are equal on a measurable space, it is enough to check that they are equal on a $\pi$-system that generates the $\sigma$-algebra and to check that the space can be covered by a countable increasing sequence of sets from the $\pi$-system, all of finite measure. See for example \cite{Cohn2013}*{Corollary 1.6.4} for a proof of this fact.

The $\pi$-system we will need in this context is the following.

\begin{lemma}\label{lemma:pi system}$ $\newline
    Let $E$ be a second-countable topological graph. Then the collection
    \begin{equation}\label{eqn:pi system}
        \mathcal{C} = \{ Z(V) \subset \partial E \ | \ V \subset E^n \text{ is a Borel subset and } n \in \mathbb{N}_0 \}
    \end{equation}
    is a $\pi$-system that generate the Borel $\sigma$-algebra of $\partial E$.
\end{lemma}
\begin{proof}$ $\newline
    Let $U\subset E^*$ be open and $K\subset E^*$ be compact. It is enough to show that $Z(U) \setminus Z(K)$ is in the $\sigma$-algebra generated by $\mathcal{C}$ to conclude that $\mathcal{C}$ generates the Borel $\sigma$-algebra of $\partial E$. For $n\in\mathbb{N}_0$ we define $U_n = U \cap E^n$. The following equality shows that $Z(U)$ is in the $\sigma$-algebra generated by $\mathcal{C}$: $Z(U) = Z(\bigcup_{n\in\mathbb{N}_0} U_n) = \bigcup_{n\in\mathbb{N}_0} Z(U_n)$. We can similarly conclude that $Z(K)$ is in the $\sigma$-algebra generated by $\mathcal{C}$. Hence, $\mathcal{C}$ generates the Borel $\sigma$-algebra of $\partial E$.

    Now we need to show that $\mathcal{C}$ is indeed a $\pi$-system. Let $U \subset E^n$ and $V \subset E^m$ be Borel subsets. We may assume that $n \geq m$. Define $W = \{ a \in U \ | \ a(m) \in V \}$. By \cref{lemma:cont of (n)} we have that $W$ is a Borel subset of $E^n$ so $Z(W) \in \mathcal{C}$. Clearly $Z(U) \cap Z(V) = Z(W)$ showing that $\mathcal{C}$ is a $\pi$-system.
\end{proof}

\begin{lemma}\label{lemma:pushforward is injective}$ $\newline
    Let $E$ be a second-countable topological graph and $\beta\in\mathbb{R}$. The pushforward $r_*:\mathcal{M}^\beta_{\text{quasi}}(E)\rightarrow\mathcal{M}^\beta_{\text{sub}}(E)$ induced by the range map $r:\partial E \rightarrow E^0$ is injective.
\end{lemma}
\begin{proof}$ $\newline
    Fix $\nu_1,\nu_2\in\mathcal{M}^\beta_{\text{quasi}}(E)$ such that $r_*\nu_1 = r_*\nu_2$. We want to check that $\nu_1 = \nu_2$ on the $\pi$-system $\mathcal{C}$ in the previous Lemma given by \cref{eqn:pi system}. To that effort, suppose that $V\subset E^n$ is a Borel subset and let
    \begin{equation*}
        \mathcal{U} = \{ U \subset E^n \ | \ U \text{ is open in } E^n \text{ and } \sigma^n|_U \text{ is injective} \}.
    \end{equation*}
    This is an open cover of $V$ since $\sigma^n$ is a local homeomorphism, furthermore since $E^n$ is second-countable we have that $\mathcal{U}$ admits a countable subcover, $\{U_i\}_{i\in\mathbb{N}}$. By partitioning $V$ with the $U_i$'s it is enough to check that $\nu_1(Z(U_i)) = \nu_2(Z(U_i))$ for each $i\in\mathbb{N}$ to conclude that $\nu_1(Z(V)) = \nu_2(Z(V))$. By definition of being $\beta$-quasi-invariant measures we get that
    \begin{align*}
        \nu_1(Z(U_i)) 
        & = e^{-\beta n} \nu_1(\sigma^n(Z(U_i))) \\
        & = e^{-\beta n} \nu_1(Z(\sigma^n(U_i))) \\
        & = e^{-\beta n} (r_*\nu_1)(\sigma^n(U_i)) \\
        & = e^{-\beta n} (r_*\nu_2)(\sigma^n(U_i)) \\
        & = \nu_2(Z(U_i)).
    \end{align*}
    Hence, $\nu_1 = \nu_2$ on the $\pi$-system $\mathcal{C}$. To conclude that $\nu_1 = \nu_2$ on the Borel $\sigma$-algebra of $\partial E$ we need to show that we can cover $\partial E$ with a sequence of increasing sets from the $\pi$-system $\mathcal{C}$, with each set having finite measure. We can do this as follows: consider the collection
    \begin{equation*}
        \mathcal{V} = \{ V\subset E^0 \ | \ V \text{ is relatively compact in } E^0 \}.
    \end{equation*}
    Since $E^0$ is locally compact Hausdorff this is an open cover of $E^0$, and since $E^0$ is second-countable $\mathcal{V}$ admits a countable subcover $\{V_i\}_{i\in\mathbb{N}}$. Define $W_i = \bigcup_{j=1}^{i}V_j$. Then each $W_i$ is relatively compact in $E^0$ and $W_i\subset W_{i+1}$, so we have that the sequence $\{Z(W_i)\}_{i\in\mathbb{N}}$ is an increasing sequence such that $\bigcup_{i\in\mathbb{N}} Z(W_i) = Z\left( \bigcup_{i\in\mathbb{N}} W_i \right) = Z(E^0) = \partial E$. Finally, we have that $\nu_1(Z(W_i)) \leq \nu_1(Z(\overline{W_i})) < \infty$ since $\nu_1$ is a regular measure and $Z(\overline{W_i})$ is compact by \cref{lemma:topgraph3}. Hence, $\nu_1 = \nu_2$, which proves that $r_*$ is injective.
\end{proof}

Proving surjectivity of the pushforward is not as straight forward as proving injectivity. The $\beta = 0$ case is done by Schafhauser, c.f.~\cite{Schafhauser2018}*{Proposition 4.4}, and his proof generalizes to all values of $\beta$ quite easily. To use his method we need to define the \emph{$n$-th boundary path space} $\partial E_n$, which for a topological graph $E$ is defined to be the following disjoint union: $\partial E_n = E^0_{\text{sng}}\sqcup \dots \sqcup E^{n-1}_{\text{sng}} \sqcup E^n$. For subsets $S\subset E^0 \sqcup ... \sqcup E^n$ we define the set $Z_n(S) = \{ a \in \partial E_n \ | \ a(k) \in S \text{ for some } k \text{ satisfying } 0\leq k \leq |a| \}$. The $n$-th boundary path space becomes a locally compact Hausdorff space with the topology generated by sets $Z_n(U)\setminus Z_n(K)$ where $U\subset E^0 \sqcup ... \sqcup E^n$ is an open subset and $K\subset E^0 \sqcup ... \sqcup E^n$ is compact. The proof of this is completely the same as the one for $\partial E$. Of course if $E$ is second-countable, $\partial E_n$ becomes second-countable as well. Note that for $n=0$ we have that $\partial E_0 = E^0$.

For $n \geq 1$ we define maps $\rho_n:\partial E_n \rightarrow \partial E_{n-1}$ by
\begin{equation*}
    \rho_n(a) =
    \left\{
    \begin{matrix*}[l]
        a(n-1) , & |a| = n, \\
        a , & |a| < n,
    \end{matrix*}
    \right.
\end{equation*}
as well as maps $\rho_{n,\infty}:\partial E \rightarrow \partial E_n$ by
\begin{equation*}
    \rho_{n,\infty}(a) =
    \left\{
    \begin{matrix*}[l]
        a(n) , & |a| > n, \\
        a , & |a| \leq n.
    \end{matrix*}
    \right.
\end{equation*}
Note that $\rho_{0,\infty}$ is the range map $r:\partial E \rightarrow E^0$. We may also apply the backwards shift map to the $n$-th boundary path space, and it will still be a local homeomorphism in this case.

\begin{lemma}{\cite{Schafhauser2018}*{Proposition 3.3}}\label{lemma:limit}$ $\newline
    Let $E$ be a topological graph. Then the maps $\rho_n$ and $\rho_{n,\infty}$ are continuous, proper and surjective. We further have that the maps $\rho_{n,\infty}$ induces a homeomorphism $\partial E \rightarrow \lim(\partial E_n, \rho_n)$. \qed
\end{lemma}

\begin{lemma}\label{lemma:surjectivity of pushforward}$ $\newline
    Let $E$ be a second-countable topological graph and $\beta\in\mathbb{R}$. The pushforward $r_*:\mathcal{M}^\beta_{\text{quasi}}(E)\rightarrow\mathcal{M}^\beta_{\text{sub}}(E)$ induced by the range map $r:\partial E \rightarrow E^0$ is surjective.
\end{lemma}
\begin{proof}$ $\newline
    Let $\mu\in\mathcal{M}^\beta_{\text{sub}}(E^0)$. We want to define regular Borel measures $\mu_n$ on $\partial E_n$ for each $n\in\mathbb{N}_0$ such that $(\rho_n)_*\mu_n = \mu_{n-1}$. We will then use universality of the limit to show that there exists a unique regular Borel measure $\nu$ on $\partial E$ such that $r_*\nu = \mu$.
    
    First we set $\mu_0 = \mu$ as $\partial E_0 = E^0$. We recursively define $\mu_n$ by the equation
    \begin{equation*}
        \mu_n = e^{-\beta} \sigma^*\mu_{n-1} + \mu|_{E^0_{\text{sng}}} - e^{-\beta}(T\mu)|_{E^0_{\text{sng}}}.
    \end{equation*}
    For Borel subsets $B\subset \partial E_n$ we read this equation as follows:
    \begin{equation}\label{eqn:recursive measures}
        \mu_n(B) = e^{-\beta}(\sigma^*\mu_{n-1})(B \setminus E^0) + \mu(B\cap E^0_{\text{sng}}) - e^{-\beta}(T\mu)(B\cap E^0_{\text{sng}}).
    \end{equation}
    Since the backwards shift map is a local homeomorphism we get that the pullback $\sigma^*\mu_{n-1}$ becomes a regular Borel measure on $\partial E_n \setminus E^0$. Furthermore, since $\mu\in\mathcal{M}^\beta_{\text{sub}}(E^0)$ we have that $\mu - e^{-\beta}(T\mu)$ is a regular Borel measure by the inequality $T\mu\leq e^\beta\mu$. Hence, $\mu_n$ is a regular Borel measure.

    We will prove by induction that $(\rho_n)_*\mu_n = \mu_{n-1}$. For $n=1$ we have that $\rho_1: \partial E_1\rightarrow E^0$ is the map $r\sqcup id_{E^0}:E^1\sqcup E^0_{\text{sng}} \rightarrow E^0$ and $\sigma:\partial E_1 \setminus E^0 \rightarrow E^0$ is the source map $s:E^1\rightarrow E^0$. Hence,
    \begin{align*}
        (\rho_1)_*\mu_1
        & = e^{-\beta}(r_*s^*\mu) + \mu|_{E^0_{\text{sng}}} - e^{-\beta}(T\mu)|_{E^0_{\text{sng}}} \\
        & = e^{-\beta}(T\mu) + \mu|_{E^0_{\text{sng}}} - e^{-\beta}(T\mu)|_{E^0_{\text{sng}}} \\
        & = e^{-\beta}(T\mu)|_{E^0_{\text{reg}}} + \mu|_{E^0_{\text{sng}}} \\
        & = \mu|_{E^0_{\text{reg}}} + \mu|_{E^0_{\text{sng}}} \\
        & = \mu \\
        & = \mu_0.
    \end{align*}
    So assume that $(\rho_k)_*\mu_k = \mu_{k-1}$ for all $k\leq n$. To show that $(\rho_{n+1})_*\mu_{n+1} = \mu_{n}$ we will consider the collection $\mathcal{C}_n = \{ Z_n(V) \subset \partial E^n \ | \ V \subset E^k \text{ is a Borel subset and } 0 \leq k \leq n \}$. Just as in \cref{lemma:pi system} this is a $\pi$-system that generates the Borel $\sigma$-algebra of $\partial E_n$. The case with $V\subset E^0$ is special, so we will tackle this last.
    
    Fix a Borel subset $V\subset E^k$ where $1 \leq k \leq n$. Note that for the set $Z_n(V)$ \cref{eqn:recursive measures} simplifies to
    \begin{equation*}\label{eqn:quasi-inv}
        \mu_{n}(Z_n(V)) = e^{-\beta}(\sigma^*\mu_{n-1})(Z_n(V)).
    \end{equation*}
    Consider the family $\mathcal{U}_k = \{ U \subset E^k \ | \ U \text{ is open in } E^k \text{ and } \sigma|_U \text{ is injective} \}$. Since $\sigma$ is a local homeomorphism on $E^k$ we get that $\mathcal{U}_k$ is an open cover of $E^k$, in particular it is an open cover of $V$. Since $E^k$ is second-countable we get that $\mathcal{U}_k$ admits a countable subcover of $V$, denote it by $\{U_i\}_{i\in\mathbb{N}}$. By partitioning $V$ we get that $(\rho_{n+1})_*\mu_{n+1}(Z_n(V)) = \mu_{n}(Z_n(V))$ if $(\rho_{n+1})_*\mu_{n+1}(Z_n(U_i)) = \mu_{n}(Z_n(U_i))$ for all $i\in\mathbb{N}$. So we compute:
    \begin{align*}
        (\rho_{n+1})_*\mu_{n+1}(Z_n(U_i))
        & = \mu_{n+1}(\rho_{n+1}^{-1}(Z_n(U_i))) \\
        & = \mu_{n+1}((Z_{n+1}(U_i))) \\
        & = e^{-\beta}(\sigma^*\mu_{n})(Z_{n+1}(U_i)) \\
        & = e^{-\beta}\mu_{n}(\sigma(Z_{n+1}(U_i))) \\
        & = e^{-\beta}\mu_{n}(Z_{n}(\sigma(U_i))) \\
        & = e^{-\beta}\mu_{n}(\rho_{n}^{-1}(Z_{n-1}(\sigma(U_i)))) \\
        & = e^{-\beta}(\rho_n)_*\mu_{n}(Z_{n-1}(\sigma(U_i))) \\
        & = e^{-\beta} \mu_{n-1}(Z_{n-1}(\sigma(U_i))) \\
        & = e^{-\beta} \mu_{n-1}(\sigma(Z_{n}(U_i))) \\
        & = e^{-\beta} (\sigma^*\mu_{n-1})(Z_{n}(U_i)) \\
        & = \mu_n(Z_n(U_i)).
    \end{align*}
    This calculation might seem tedious, but it is nice to exemplify \emph{how} these types of calculations are done, as they aren't the easiest to manoeuvre.

    When $V\subset E^0$ we have that $Z_n(V)\setminus E^0 = Z_n(r^{-1}(V))$. By partitioning $r^{-1}(V)$ with sets $U\in\mathcal{U}_1$ we may assume that $\sigma|_{r^{-1}(V)}$ is injective. The same calculation as above can now be carried out, where we also have to drag along the terms involving $\mu|_{E^0_{\text{sng}}}$ and $(T\mu)|_{E^0_{\text{sng}}}$. So we have that $(\rho_{n+1})_*\mu_{n+1} = \mu_{n}$ on the $\pi$-system $\mathcal{C}_n$. Making the same argument as we did in \cref{lemma:pushforward is injective}, we get that $(\rho_{n+1})_*\mu_{n+1} = \mu_{n}$ on $\partial E_n$.

    The rest of the proof follows in exactly the same manner as in \cite{Schafhauser2018}*{Proposition 4.4}.
\end{proof}

\begin{proposition}\label{prop:pushforward bijective}$ $\newline
    Let $E$ be a second-countable topological graph and $\beta\in\mathbb{R}$. The pushforward $r_*:\mathcal{M}^\beta_{\text{quasi}}(E)\rightarrow\mathcal{M}^\beta_{\text{sub}}(E)$ induced by the range map $r:\partial E \rightarrow E^0$ is an affine bijection.
\end{proposition}
\begin{proof}$ $\newline
    The fact that the pushforward is bijective is the combination of \cref{lemma:pushforward is injective} and \cref{lemma:surjectivity of pushforward}. Linearity of the pushforward gives us that this is an affine bijection.
\end{proof}

We are now able to state and prove the main result of this paper.

\begin{theorem}\label{thm:KMS weights 1:1 sub-inv measures}$ $\newline
    Let $E$ be a second-countable topological graph, $\beta\in\mathbb{R}\setminus\{0\}$. There is an affine bijection between the following three sets: KMS$_\beta$-weights (states) for the gauge-action on $C^*(E)$, (probability measures in) $\mathcal{M}^\beta_{\text{quasi}}(E)$ and (probability measures in) $\mathcal{M}^\beta_{\text{sub}}(E)$.
\end{theorem}
\begin{proof}$ $\newline
    The first bijection follows by \cref{prop:KMS weights graph} and the second bijection follows by \cref{prop:pushforward bijective}. For the KMS$_\beta$-\emph{states} case, we only need to check that the pushforward of a probability measure $\nu \in \mathcal{M}^\beta_{\text{quasi}}(E)$ is a probability measure. Indeed, by surjectivity of the range map, we get that
    \begin{equation*}
        r_*\nu(E^0) = \nu(r^{-1}(E^0)) = \nu(\partial E) = 1. \qedhere
    \end{equation*}
\end{proof}

It is worth noting that in the case where $E$ is a countable discrete graph, \cref{thm:KMS weights 1:1 sub-inv measures} specializes to a Theorem by Thomsen, c.f.~\cite{Thomsen2014}*{Theorem 4.6}.

\begin{example}\label{example:no KMS states but KMS weights}$ $\newline
    Consider the doubly infinite discrete graph
    \begin{equation*}
        \xymatrix{
            \cdots 
            & \bullet_{v_{-2}} \ar[l]^{e_{-2}}
            & \bullet_{v_{-1}} \ar[l]^{e_{-1}}
            & \bullet_{v_{0}} \ar[l]^{e_{0}}
            & \bullet_{v_{1}} \ar[l]^{e_{1}}
            & \bullet_{v_{2}} \ar[l]^{e_{2}}
            & \cdots \ar[l]^{e_{3}}
        }
    \end{equation*}
    Any regular Borel measure on $E^0 = \{v_n\}_{n\in\mathbb{Z}}$ is of the form
    \begin{equation*}
        \sum_{n\in\mathbb{Z}} a_n \delta_{v_n}
    \end{equation*}
    with $0 \leq a_n < \infty$. Note also that $E^0 = E^0_{\text{reg}}$. Thus, for $\beta \in \mathbb{R}$ we are interested in the measures $\mu$ on $E^0$ such that $T\mu = e^{\beta}\mu$. It should be clear that $T\delta_{v_n} = \delta_{v_{n-1}}$, hence we are interested in the measures on $E^0$ such that
    \begin{equation*}
        \sum_{n\in\mathbb{Z}} a_n \delta_{v_{n-1}} 
        = \sum_{n\in\mathbb{Z}} a_n e^{\beta} \delta_{v_n}.
    \end{equation*}
    This is the case if and only if $a_{n} = a_{n-1}e^{\beta}$. Thus,
    \begin{equation*}
        \mathcal{M}_{\text{sub}}^{\beta}(E) 
        = \left\{ a \sum_{n\in\mathbb{Z}} e^{n\beta} \delta_{v_n} \ | \  0 \leq a < \infty \right\}.
    \end{equation*}
    It should be clear that there are no values of $a$ or $\beta$ that make
    \begin{equation*}
        a \sum_{n\in\mathbb{Z}} e^{n\beta} \delta_{v_n}
    \end{equation*}
    a probability measure. Hence, there are no KMS$_\beta$-states for the gauge action on the $C^*$-algebra of this graph, but there is this one-dimensional ray of KMS$_\beta$-weights for the gauge action for any value of $\beta \neq 0$, by \cref{thm:KMS weights 1:1 sub-inv measures}.
\end{example}

Note that the conclusion in the following example was reached by Thomsen using different methods in the case where $Z$ is a compact metric space, c.f.~\cite{Thomsen2011}*{Example 6.14}.

\begin{example}\label{example:homeomorphism algebra}$ $\newline
    Let $Z$ be a second-countable locally compact Hausdorff space, and suppose that $h:Z\rightarrow Z$ is a surjective local homeomorphism such that $\#h^{-1}(z) = n$ for all $z\in Z$. Note that this identifies $Z$ as an $n$-covering space of $Z$. We get a second-countable topological graph by setting the vertex and edge spaces equal to $Z$, the range map is set to be the identity on $Z$ and the source map is set to $h$, i.e. we consider the topological graph $E = (Z,Z,id_Z,h)$. In this case we have that
    \begin{equation*}
        \mathcal{M}_{\text{sub}}^{\beta}(E)
        = \{ \mu \in \mathcal{M}_{\text{reg}}(Z) \ | \ h^*\mu = e^\beta\mu \}.
    \end{equation*}
    We will show that there can only exist KMS$_\beta$-states for the gauge-action on the associated graph $C^*$-algebra if $\beta = \ln(n)$. By \cref{thm:KMS weights 1:1 sub-inv measures} this is equivalent to showing that the pullback $h^*:\mathcal{M}_{\text{reg}}(Z) \rightarrow \mathcal{M}_{\text{reg}}(Z)$ only has one possible eigenvalue, namely $n$, when viewed as a linear map between real vector spaces (after taking the group completion of $\mathcal{M}_{\text{reg}}(Z)$).

    Since $Z$ is an $n$-covering space of $Z$ we get that for every $z\in Z$ there exists an open neighborhood $U_z$ about $z$ such that $h^{-1}(U_z) = \bigcup_{i=1}^{n}V_{z,i}$ with the $V_{z,i}$'s being mutually disjoint and $h|_{V_{z,i}}$ being injective. Then $\mathcal{U} = \{U_z\}_{z \in Z}$, with each $U_z$ as above, is an open cover of $Z$. By second-countability of $Z$ we have that $\mathcal{U}$ admits a countable subcover, denote it by $\{U_k\}_{k\in\mathbb{N}}$. For each $k\in\mathbb{N}$ let $V_{k,1},\dots,V_{k,n}$ be mutually disjoint open subsets of $Z$ such that $h^{-1}(U_k) = \bigcup_{i=1}^{n}V_{k,i}$ and $h|_{V_{k,i}}$ is injective for each $1 \leq i \leq n$. Let us now consider the collection
    \begin{equation*}
        \mathcal{V} = \{ V_{k,1},\dots,V_{k,n} \}_{k\in\mathbb{N}}.
    \end{equation*}

    \emph{Claim 1:} This is an open cover of $Z$. Indeed, let $z \in Z$ and choose $k\in\mathbb{N}$ such that $h(z) \in U_k$. Then $z \in h^{-1}(U_k) = \bigcup_{i=1}^{n}V_{k,i}$, hence $z \in V_{k,i}$ for some $1 \leq i \leq n$.

    \emph{Claim 2:} $h\left(\bigcup_{k\in\mathbb{N}} V_{k,j}\right) = Z$ for any $1 \leq j \leq n$. Indeed, let $z \in Z$ and choose $k\in\mathbb{N}$ such that $z \in U_k$. Then, $h^{-1}(U_k) = \bigcup_{i=1}^{n}V_{k,i}$ and, in particular, $h|_{V_{k,j}}:V_{k,j} \rightarrow U_k$ is a homeomorphism. Thus, there exists a $w \in V_{k,j}$ such that $h(w) = z$, hence $z \in h\left(\bigcup_{k\in\mathbb{N}} V_{k,j}\right)$.

    With these claims in place, suppose we have a probability measure $\mu \in \mathcal{M}_{\text{reg}}(Z)$ satisfying $h^*\mu = e^\beta \mu$. Partition $Z$ by $\mathcal{V}$, i.e. write $Z$ as a union of mutually disjoint sets $Z_{k,i}$ where each $Z_{k,i} \subset V_{k,i}$ and $\bigcup_{k\in\mathbb{N}}Z_{k,i} = \bigcup_{k\in\mathbb{N}}V_{k,i}$. By \cref{prop:pullback of measure} we have that
    \begin{equation*}
        h^*\mu (Z)
        = \sum_{i = 1}^{n} \sum_{k\in\mathbb{N}} \mu(h(Z_{k,i}))
        = \sum_{i = 1}^{n} \mu\left( h \left( \bigcup_{k\in\mathbb{N}} Z_{k,i} \right) \right)
        = \sum_{i = 1}^{n} \mu(Z)
        = n\mu(Z)
        = e^\beta \mu(Z).
    \end{equation*}
    Since $\mu(Z) = 1$ we get that $\beta = \ln(n)$.
\end{example}

It is worth noting that situations similar to this example has been studied before in many contexts. For example, Exel and Lopes in \cites{Exel2004} study the situation where $T:X\rightarrow X$ is a local homeomorphism on a compact metric space. Afsar, an Huef and Raeburn, in \cites{Afsar2014} studied the situation where $h:Z\rightarrow Z$ is a surjective local homeomorphism on a compact Hausdorff space.

\subsection*{Final remarks}

\cref{thm:KMS weights 1:1 sub-inv measures} attains a description of the KMS$_\beta$-weights for the gauge-action on the $C^*$-algebra of a second-countable topological graph. It does not however, get us any closer to the existence problem of these KMS$_\beta$-weights. A possible direction of further investigation would be to attempt to get a description of the set $\mathcal{M}^{\beta}_{\text{sub}}(E)$, in complete generality (ambitious) or for more concrete classes of second-countable topological graphs.

\end{document}